\documentclass[10pt]{article}

\usepackage{hyperref}
\usepackage{xcolor}
\usepackage{amssymb}
\usepackage{amsfonts}
\usepackage{amsthm}
\usepackage{amsmath}
\usepackage{float}
\usepackage{bigints}
\usepackage{fullpage}
\usepackage{mathtools}
\usepackage[utf8]{inputenc} 
\usepackage{cancel}
\usepackage{verbatim}

\newtheorem{theorem}{Theorem}

\newtheorem{lemma}[theorem]{Lemma}

\newtheorem{definition}{Definition}

\newtheorem{corollary}[theorem]{Corollary}
\newtheorem{remark}[theorem]{Remark}

\newcommand{\x}{\mathbf{x}}
\newcommand{\p}{\mathbf{p}}
\newcommand{\q}{\mathbf{q}}

\newcommand{\oo}{\mathbf{o}}

\newcommand{\B}{\mathbf{B}}

\newcommand{\K}{\mathbf{K}}

\newcommand{\Ee}{\mathbb{E}}

\newcommand{\Mm}{\mathbb{M}}

\newcommand{\Rr}{\mathbb{R}}

\newcommand{\Ed}{\Ee^d}
\newcommand{\Rd}{\Rr^d}

\newcommand{\Md}{\Mm^d}

\newcommand{\KK}{{\mathbf K}}
\newcommand{\iprod}[2]{\left<#1,#2\right>}

\newcommand{\noshow}[1]{}

\title{On uniform contractions of balls in Minkowski spaces
\footnote{\hspace*{.35cm} 2010 Mathematics Subject Classification: 52A20, 52A22.}}

\author{K\'{a}roly Bezdek\thanks{Partially supported by a Natural Sciences and 
Engineering Research Council of Canada Discovery Grant.}
}

\date{}

\begin{document}

\maketitle

\begin{abstract}
\noindent Let $N$ balls of the same radius be given in a $d$-dimensional real normed vector space, i.e., in a Minkowski $d$-space. Then apply a uniform contraction to the centers of the $N$ balls without changing the common radius. Here a uniform contraction is a contraction where all the pairwise distances in the first set of centers are larger than all the pairwise distances in the second set of centers. The main results of this paper state that a uniform contraction of the centers does not increase (resp., decrease) the volume of the union (resp., intersection) of $N$ balls in Minkowski $d$-space, provided that $N\geq 2^d$ (resp., $N\geq 3^d$ and the unit ball of the Minkowski $d$-space is a generating set). Some improvements are presented in Euclidean spaces. 
\end{abstract}

\section{Introduction}\label{sec:intro}

The Kneser--Poulsen Conjecture \cite{Kn}, \cite{Po} (resp., Gromov--Klee--Wagon conjecture \cite{Gr87}, \cite{Kl}, \cite{KlWa}) states that if the centers of a family of $N$ unit balls in Euclidean $d$-space is contracted, then the volume of the union (resp., intersection) does not increase (resp., decrease). These conjectures have been proved by Bezdek and Connelly \cite{BeCo} for $d=2$ (in fact, for not necessarily congruent circular disks as well) and they are open for all $d\geq 3$. For a number of partial results in dimensions $d\geq3$, we refer the interested reader to the corresponding chapter in \cite{Be}. Very recently Bezdek and Nasz\'odi \cite{BeNa} investigated the Kneser--Poulsen conjecture as well as the Gromov--Klee--Wagon conjecture for special contractions in particular, for uniform contractions.  Here, a uniform contraction is a contraction where all the pairwise distances in the first set of centers are larger than all the pairwise distances in the second set of centers. The main result of \cite{BeNa} states that a uniform contraction of the centers does not increase (resp., decrease) the volume of the union (resp., intersection) of $N$ unit balls in Euclidean $d$-space ($d\geq 3$), provided that $N\geq c^dd^{2.5d}$, where $c>0$ is a universal constant (resp., $N\geq(1+\sqrt{2})^d$). In this paper we improve these results and extend them to Minkowski spaces.

Let $\K\subset \Rd$ be an $\oo$-symmetric convex body, i.e., a compact convex set with nonempty interior symmetric about the origin $\oo$ in $\Rd$. Let $\| \cdot \|_{\K}$ denote the norm generated by $\K$, which is defined by $\|\x\|_{\KK}:=\min\{\lambda \geq 0\  |\  \lambda\x\in \K\}$ for $\x\in \Rd$. Furthermore, let us denote $\Rd$ with the norm $\| \cdot \|_{\K}$ by $\Md_{\K}$ and call it the {\it Minkowski space of dimension $d$ generated by} $\K$. We write $B_{\K}[\x, r]:=\x+r\K$ for $\x\in\Rd$ and $r>0$ and call any such set a (closed) {\it ball of radius $r$}, where $+$ refers to vector addition extended to subsets of $\Rd$ in the usual way. The following definitions introduce the core notions and notations for our paper.

\begin{definition}\label{r-neighbourhood and r-body}
For $X\subseteq\Rd$ and $r>0$ let
$$X_{r}^{\K}:=\bigcup\{B_{\K}[\x,r]\ |\ \x\in X\}\ ({\rm resp.,}\ X_{\KK}^{r}:=\bigcap\{B_{\K}[\x,r]\ |\ \x\in X\})$$
denote the {\rm $r$-ball neighbourhood} of $X$ (resp., {\rm $r$-ball body} generated by $X$) in $\Md_{\K}$. If $X\subset\Rd$ is a finite set, then we call
$X_{r}^{\K}$ (resp., $X_{\KK}^{r}$) the {\rm $r$-ball molecule} (resp., {\rm $r$-ball polyhedron}) generated by $X$ in $\Md_{\K}$.
\end{definition}

\begin{remark}
We note that $r$-ball bodies and $r$-ball polyhedra have been intensively studied (under various names) from the point of view of convex and discrete geometry in a number of publications (see the recent papers \cite{BLNP}, \cite{JMR}, \cite{KMP}, \cite{LNT}, \cite{MRS}, and the references mentioned there).
\end{remark} 

\begin{definition}\label{uniform contraction}
We say that the (labeled) point set $Q:=\{\q_1, \dots , \q_N\}\subset\Rd$ is a {\rm uniform contraction} of the (labeled) point set $P:=\{\p_1, \dots ,\p_N\}\subset\Rd$ with {\rm separating value} $\lambda>0$ in $\Md_{\K}$  if
$$\|\q_i-\q_j\|_{\K}\leq\lambda\leq\|\p_i-\p_j\|_{\K}\ {\rm holds}\ {\rm for}\  {\rm all}\ 1\leq i<j\leq N.$$
\end{definition}

In order to state the main results of this paper, let $V_d(\cdot)$ denote the Lebesgue measure in $\Rd$ (with $V_d(\emptyset)=0$).

\begin{theorem}\label{union} Let $\K$ be an $\oo$-symmetric convex body in $ \Rd$.
If $r>0, \lambda>0, d>1$, $N\geq 2^d$, and $Q:=\{\q_1, \dots , \q_N\}\subset\Rd$ is a uniform contraction of $P:=\{\p_1, \dots ,\p_N\}\subset\Rd$ with separating value 
$\lambda$ in $\Md_{\K}$, then
\begin{equation}\label{main-1}
V_d(Q_r^{\K})\leq V_d(P_r^{\K}).
\end{equation}
\end{theorem}

\begin{remark}\label{union-conv}
The proof of Theorem~\ref{union} presented below yields the following statement as well. If $r\geq\frac{\lambda}{2}>0$, $d>1$, $N\geq 2^d$, and $Q$ is a uniform contraction of $P$ with separating value $\lambda$ in $\Md_{\K}$, then 
\begin{equation}\label{main-11}
V_d\left({\rm conv}(Q_r^{\K})\right)\leq V_d\left({\rm conv}(P_r^{\K})\right),
\end{equation}
where ${\rm conv}(\cdot )$ stands for the convex hull of the given set in $\Rd$.
\end{remark}


Recall from \cite{Sc} that the compact convex set $\emptyset\neq{A'}\subset\Rd$ is a {\it summand} of the compact convex set  $\emptyset\neq{A}\subset\Rd$ if there exists a compact convex set  $\emptyset\neq{A''}\subset\Rd$ such that ${A'}+{A''}={A}$. Furthermore, following \cite{MoSc} we say that the convex body $\B\subset\Rd$ is a {\it generating set} if any nonempty intersection of translates of $\B$ is a summand of $\B$. In particular, we say that {\it $\Md_{\K}$ possesses a generating unit ball} if $B_{\K}[\oo, 1]=\K$ is a generating set in $\Rd$. For a recent overview on generating sets see the relevant subsections in \cite{MS} and \cite{MoSc}. Here we recall the following statements only. Two-dimensional convex bodies are generating sets. Euclidean balls are generating sets as well and the system of generating sets is stable under non-degenerate linear maps and under direct sums. Furthermore, a centrally symmetric convex polytope is a generating set if and only if it is a direct sum of convex polygons and in odd dimension, a line segment.

\begin{theorem}\label{intersection}
Let $r>0, \lambda>0, d>1$, $N\geq 3^d$, and let the $\oo$-symmetric convex body $\K$ be a generating set in $\Rd$. If $Q:=\{\q_1, \dots , \q_N\}\subset\Rd$ is a uniform contraction of $P:=\{\p_1, \dots ,\p_N\}\subset\Rd$ with separating value $\lambda$ in $\Md_{\K}$, then
\begin{equation}\label{main-2}
V_d(P_{\K}^{r})\leq V_d(Q_{\K}^{r}).
\end{equation}
\end{theorem}

\begin{remark}\label{intersection-remark}
We say that {\rm the balls of $\Md_{\K}$ are volumetric maximizers for $r$-ball bodies in $\Md_{\K}$} if for any compact set $\emptyset\neq A\subset\Rd$ with $V_d(A)>0$ the inequality
\begin{equation}\label{Blaschke-Santalo}
V_d(A_{\K}^{r})\leq V_d(B_{\K}[\oo, r-r_{\K}(A)])
\end{equation}
holds for all $r>r_{\K}(A)$, where $V_d(A)=V_d(B_{\K}[\oo, r_{\K}(A)])$. On the one hand, if the balls of $\Md_{\K}$ are generating sets in $\Rd$, then they are volumetric maximizers for $r$-ball bodies in $\Md_{\K}$. On the other hand, if the balls of $\Md_{\K}$ are volumetric maximizers for $r$-ball bodies in $\Md_{\K}$, then (\ref{main-2}) holds
whenever $Q:=\{\q_1, \dots , \q_N\}\subset\Rd$ is a uniform contraction of $P:=\{\p_1, \dots ,\p_N\}\subset\Rd$ with separating value $\lambda$ in $\Md_{\K}$ and $r>0, \lambda>0, d>1$, $N\geq 3^d$. Thus, it would be interesting to find a proper characterization of those Minkowski spaces $\Md_{\K}$ whose balls are volumetric maximizers for $r$-ball bodies, that is, for which
(\ref{Blaschke-Santalo}) holds.
\end{remark}

\noindent We simplify our notations when $\K$ is a Euclidean ball of $\Rd$ as follows. We denote the Euclidean norm of a vector $\p$ in the $d$-dimensional Euclidean space $\Ee^d$ by $|\p|:=\sqrt{\iprod{\p}{\p}}$, where $\iprod{\cdot}{\cdot}$ is the standard inner product. The closed Euclidean ball of radius $r$ centered at the point $\p\in\Ed$ is denoted by $\B^d[\p,r]:=\{\q\in\Ed\ |\  |\p-\q|\leq r\}$. For a set $ X\subseteq\Ee^d$, $d>1$ and $r>0$, let $X_r:=\bigcup_{\x\in X}\B^d[\x, r]$ (resp., $X^r:=\bigcap_{\x\in X}\B^d[\x, r]$). Let $\emptyset\neq A\subset\Ed$ be a compact convex set, and $0\leq k\leq d$. We denote the {\it $k$-th quermassintegral} of $A$ by $W_k(A)$. It is well known that $W_0(A)=V_d(A)$. Moreover, $dW_1(A)$ is the surface area of $A$, $\frac{2}{\omega_d}W_{d-1}(A)$ is equal to the mean width of $A$, and $W_d(A)=\omega_d$, where $\omega_d$ stands for the volume of a $d$-dimensional unit ball, that is, $\omega_d=\frac{\pi^{\frac{d}{2}}}{\Gamma(1+\frac{d}{2})}$ (\cite{Sc}, p. 290-291). In this paper, for simplicity $W_k({\emptyset})=0$ for all $0\leq k\leq d$. Here we recall Kubota's integral recursion formula (\cite{Sc}, p. 295), according to which 
\begin{equation}\label{Kubota}
W_k(A)=\frac{1}{d\omega_{d-1}}\int_{S^{d-1}} W_{k-1}\left(P_{{\bf u}^\perp}(A)\right)d\lambda({\bf u})
\end{equation}
holds for any compact convex set $\emptyset\neq A\subset\Ed$ and for any $0<k<d$, where $S^{d-1}:={\rm bd}(\B^d[\oo, 1])=\{\x \in\Ee^d\ |\ |\x |=1\}$, $d\lambda(\cdot)$ is the spherical Lebesgue measure on $S^{d-1}$, and $P_{{\bf u}^\perp}(\cdot)$ is the orthogonal projection onto the orthogonal complement of the $1$-dimensional linear subspace spanned by ${\bf u}\in S^{d-1}$. Finally, we recall that Ohmann \cite{Oh52}, \cite{Oh54}, \cite{Oh56} using Kubota's formula (\ref{Kubota}) has inductively defined the quermassintegrals $W_k(A)$, $0<k<d$ for any compact set $\emptyset\neq A\subset\Ed$ with $W_0(A):=V_d(A)$ and $W_d(A):=\omega_d$ and proved analogues of some classical inequalities on quermassintegrals. In what follows we use Ohmann's extension of the classical quermassintegrals for non-convex compact sets.

We note that if $\K$ is a Euclidean ball in $\Rd$, then Theorem~\ref{union} improves Theorem 1.5 of \cite{BeNa} by replacing the condition $N\geq c^dd^{2.5d}$ with the weaker condition $N\geq 2^d$. On the other hand, Theorem 1.4 of \cite{BeNa} improves Theorem~\ref{intersection} for $\Md_{\K}=\Ee^d$ as follows:  if $\lambda>0$, $r>0$, $d>1$, $0\leq k<d$, $ N\geq \left(1+\sqrt{2}\right)^d=(2.414\dots)^d$, and $Q:=\{\q_1,\dots ,\q_N\}\subset\Ee^d$ is a uniform contraction of $P:=\{\p_1,\dots ,\p_N\}\subset\Ee^d$ with separating value $\lambda$ in $\Ee^d$, then $W_k(P^r)\leq W_k(Q^r)$ (see also \cite{Be19}). In this paper, we improve the later result for $k=0$ in large dimensions moreover, extend Theorem~\ref{union} and Remark~\ref{union-conv} to intrinsic volumes when $\K$ is a Euclidean ball in $\Rd$. 

\begin{theorem}\label{Euclidean}
\item{\bf (i)} If $r\geq\frac{\lambda}{2}>0$, $0\leq k<d$, $N\geq 2^d$, and $Q:=\{\q_1, \dots , \q_N\}\subset\Ee^d$ is a uniform contraction of $P:=\{\p_1, \dots ,\p_N\}\subset\Ee^d$ with separating value $\lambda$ in $\Ee^d$, then 
\begin{equation}\label{main-22-conv-Euclidean}
W_k\left({\rm conv}(Q_r)\right)\leq W_k\left({\rm conv}(P_r)\right).
\end{equation}
and
\begin{equation}\label{main-22-Euclidean}
W_k\left(Q_r\right)\leq W_k\left(P_r\right).
\end{equation}
\item{\bf (ii)} If $r>0, \lambda>0$, $d\geq d_0$ (with a (large) universal constant $d_0$), $N\geq {2.359}^d$, and $Q:=\{\q_1, \dots , \q_N\}\subset\Ee^d$ is a uniform contraction of $P:=\{\p_1, \dots ,\p_N\}\subset\Ee^d$ with separating value $\lambda$ in $\Ee^d$, then
\begin{equation}\label{main-2-Euclidean}
V_d(P^{r})\leq V_d(Q^{r}).
\end{equation}
\end{theorem}

\begin{remark}
It has been proved in \cite{Al}, \cite{CaPa}, and \cite{Su} that the mean width of the convex hull of a finite subset of $\Ee^d$, $d>1$ is not less than the mean width of the convex hull of any of its contractions in $\Ee^d$. From this it follows in a straightforward way that if $r>0$, $d>1$, $N>1$, and $Q:=\{\q_1, \dots , \q_N\}$ is a contraction of $P:=\{\p_1, \dots ,\p_N\}$ in $\Ee^d$, then 
$W_{d-1}\left({\rm conv}(Q_r)\right)\leq W_{d-1}\left({\rm conv}(P_r)\right)$. Thus, it is natural to ask whether also $W_{d-1}(Q_r)\leq W_{d-1}(P_r)$ holds whenever $Q:=\{\q_1, \dots , \q_N\}$ is a contraction of $P:=\{\p_1, \dots ,\p_N\}$ in $\Ee^d$ with $r>0$, $d>1$, and $N>1$. This question for $d=2$ can be regarded as a (somewhat unusual) relative of Alexander's longstanding conjecture \cite{Al} (see also \cite{BeCoCs}), which states that if $Q:=\{\q_1, \dots , \q_N\}$ is a contraction of $P:=\{\p_1, \dots ,\p_N\}$ in $\Ee^2$, then $W_1(P^r)\leq W_1(Q^r)$ holds for $r>0$ and $N>1$.
\end{remark}


In the rest of the paper we prove the theorems stated.

\section{Proof of Theorem~\ref{union}}

As (\ref{main-1}) holds trivially for $0<r\leq \frac{\lambda}{2}$ therefore we may assume that $0<\frac{\lambda}{2}<r$. Recall that for a bounded set  $\emptyset\neq X\subset\Rd$
the diameter ${\rm diam}_{\K}(X)$ of $X$ in $\Md_{\K}$ is defined by ${\rm diam}_{\K}(X):=\sup\{\|\x_1-\x_2\|_{\K} \ | \ \x_1,\x_2\in X\}$. Clearly,
\begin{equation}\label{union-1}
{\rm diam}_{\K}(Q_r^{\K})={\rm diam}_{\K}(Q)+2r\leq \lambda+2r .
\end{equation}
Thus, the isodiametric inequality in Minkowski spaces (Theorem 11.2.1 in \cite{BuZa}) and (\ref{union-1}) imply that
\begin{equation}\label{union-2}
V_d(Q_r^{\K})\leq \left(r+\frac{\lambda}{2}\right)^dV_d(\K) .
\end{equation}

\noindent For the next estimate recall that the {\it volumetric radius relative to $\K$} of the compact set $\emptyset\neq A\subset\Rd$ is denoted by $r_{\K}(A)$ and it is defined by $V_d\left(r_{\K}(A)\K\right)=\left(r_{\K}(A)\right)^dV_d(\K):=V_d(A)$. Using this concept one can derive the following inequality from the Brunn--Minkowski inequality in a rather straightforward way (Theorem 9.1.1 in \cite{BuZa}):
\begin{equation}\label{union-3}
r_{\K}(A_{\epsilon}^{\K})\geq r_{\K}(A)+\epsilon ,
\end{equation}
which holds for any $\epsilon>0$. As $\{B_{\K}\left[\p_i, \frac{\lambda}{2}\right]\ |\ 1\leq i\leq N\}$ is a packing in $\Rd$ therefore $r_{\K}\left(P_{\frac{\lambda}{2}}^{\K}\right)=N^{\frac{1}{d}}\frac{\lambda}{2}$. Combining this observation with (\ref{union-3}) yields

\begin{equation}\label{union-4}
V_d(P_r^{\K})=V_d\left( \left(P_{\frac{\lambda}{2}}^{\K}\right)_{r-\frac{\lambda}{2}}^{\K}\right)\geq\left(N^{\frac{1}{d}}\frac{\lambda}{2}+(r-\frac{\lambda}{2}) \right)^dV_d(\K)=\left(r+(N^{\frac{1}{d}}-1)\frac{\lambda}{2} \right)^dV_d(\K) .
\end{equation}
Finally, as $N\geq 2^d$ therefore $N^{\frac{1}{d}}-1\geq 1$ and Theorem~\ref{union} follows from (\ref{union-2}) and (\ref{union-4}) in a straightforward way.

\section{Proof of Remark~\ref{union-conv}}

By assumption $0<\frac{\lambda}{2}\leq r$. Furthermore, we clearly have ${\rm diam}_{\K}\left({\rm conv}(Q_r^{\K})\right)={\rm diam}_{\K}\left(\left({\rm conv}(Q)\right)_r^{\K}\right)= {\rm diam}_{\K}\left({\rm conv}(Q)\right)+2r={\rm diam}_{\K}(Q)+2r\leq \lambda+2r$. Thus, the isodiametric inequality (Theorem 11.2.1 in \cite{BuZa}) applied to ${\rm conv}(Q_r^{\K})$ yields
\begin{equation}\label{union-22}
V_d\left({\rm conv}(Q_r^{\K})\right)\leq \left(r+\frac{\lambda}{2}\right)^dV_d(\K) .
\end{equation}
On the other hand, as ${\rm conv}(P_r^{\K})\supseteq P_r^{\K}$ therefore (\ref{union-4}) yields
\begin{equation}\label{union-44}
V_d\left({\rm conv}(P_r^{\K})\right)\geq\left(r+(N^{\frac{1}{d}}-1)\frac{\lambda}{2} \right)^dV_d(\K) .
\end{equation}    
Finally, $N\geq 2^d$, (\ref{union-22}) and (\ref{union-44}) complete the proof of Remark~\ref{union-conv}.

\section{Proof of Theorem~\ref{intersection}}

The following proof extends the core ideas of the proof of Theorem 1.4 from \cite{BeNa} to Minkowski spaces. For a bounded set $\emptyset\neq X\subset \Rd$ let ${\rm cr}_{\K}(X):=\inf \{R>0 \ |\ X\subseteq B_{\K}[\x, R]\ {\rm with}\ \x\in\Rd\}$. We call ${\rm cr}_{\K}(X)$ the {\it circumradius} of $X$ in $\Md_{\K}$. Now, recall that $P=\{\p_1,\dots ,\p_N\}\subset\Rd$ with $N\geq 3^d$ such that $\lambda\leq \|\p_i-\p_j\|_{\K}$ holds for all $1\leq i<j\leq N$. We claim that
\begin{equation}\label{intersection-1}
\lambda\leq {\rm cr}_{\K}(P) .
\end{equation}
For a proof assume that ${\rm cr}_{\K}(P)<\lambda$. Then there exists $\x_0\in\Rd$ such that
\begin{equation}\label{indirect-1}
P_{\frac{\lambda}{2}}^{\K}\subset B_{\K}\left[\x_0, \frac{3}{2}\lambda\right] .
\end{equation}
As $\{B_{\K}\left[\p_i, \frac{\lambda}{2}\right]\ |\ 1\leq i\leq N\}$ is a packing in $\Rd$ therefore
\begin{equation}\label{indirect-2}
V_d\left(P_{\frac{\lambda}{2}}^{\K}\right)=N\left(\frac{\lambda}{2}\right)^dV_d(\K) .
\end{equation}
Finally, (\ref{indirect-1}) and (\ref{indirect-2}) imply that $N\left(\frac{\lambda}{2}\right)^dV_d(\K)<\left(\frac{3}{2}\lambda\right)^dV_d(\K)$ and therefore $N<3^d$, a contradiction.
This completes the proof of (\ref{intersection-1}).

If $r\leq {\rm cr}_{\K}(P)$, then clearly $V_d(P_{\K}^r)=V_d(\emptyset)=0\leq V_d(Q_{\K}^r)$, finishing the proof of Theorem~\ref{intersection} in this case.

Hence, for the rest of the proof of Theorem~\ref{intersection} we may assume via (\ref{intersection-1}) that
\begin{equation}\label{intersection-2}
0<\lambda\leq {\rm cr}_{\K}(P)<r .
\end{equation}
Next, recall that $Q=\{\q_1,\dots ,\q_N\}\subset\Rd$ with $N\geq 3^d$ such that $\|\q_i-\q_j\|_{\K}\leq \lambda$ holds for all $1\leq i<j\leq N$. Thus, Bohnenblust's theorem (Theorem 11.1.3 in \cite{BuZa}) yields ${\rm cr}_{\K}(Q)\leq \frac{d}{d+1}{\rm diam}_{\K}(Q)\leq\frac{d}{d+1}\lambda$, from which it is easy to derive that
\begin{equation}\label{intersection-3}
V_d(Q_{\K}^r)\geq \left(r-\frac{d}{d+1}\lambda\right)^dV_d(\K) .
\end{equation}
Here (\ref{intersection-2}) guarantees that $r-\frac{d}{d+1}\lambda>r-\lambda> 0$.

For a bounded set $\emptyset\neq X\subset\Rd$ and $r>0$ with ${\rm cr}_{\K}(X)\leq r$ let ${\rm conv}_{r, \K}(X):=\bigcap\{B_{\K}[\x, r]\ |\ X\subseteq B_{\K}[\x , r]\ {\rm with}\ \x\in\Rd\}$.
We call ${\rm conv}_{r, \K}(X)$ the {\it $r$-ball convex hull} of $X$ in $\Md_{\K}$. If $\emptyset\neq X\subset\Rd$ is a bounded set and $r>0$ with ${\rm cr}_{\K}(X)> r$, then let ${\rm conv}_{r, \K}(X):=\Rd$. Moreover, for an unbounded set $X\subseteq\Rd$ and $r>0$ let ${\rm conv}_{r, \K}(X):=\Rd$. Furthermore, for simplicity let ${\rm conv}_{r, \K}(\emptyset):=\emptyset$. Finally, we say that $X\subseteq\Rd$ is {\it $r$-ball convex} for $r>0$ in $\Md_{\K}$ if $X={\rm conv}_{r, \K}(X)$. Clearly, $X_{\K}^r$ is $r$-ball convex in $\Md_{\K}$ for any $X\subseteq\Rd$. 

\begin{lemma}\label{special}
Let $d>1$and $r>0$ be given and let $\Md_{\K}$ possess a generating unit ball. If $X_{\K}^r\neq\emptyset$, then
\begin{equation}\label{intersection-5}
X_{\K}^r-{\rm conv}_{r, \K}(X)=B_{\K}[\oo, r] .
\end{equation}
\end{lemma}

\begin{proof}
Clearly, as $B_{\K}[\oo, 1]=\K$ is a generating set in $\Rd$ therefore the closed ball having radius $r>0$ in $\Md_{\K}$, i.e., $B_{\K}[\oo , r]=r\K$ is also a generating set in $\Rd$. In particular, $X_{\K}^r\neq\emptyset$ is a summand of $B_{\K}[\oo , r]$. Now, recall Lemma 3.1.8 of \cite{Sc} stating that the compact convex set $\emptyset\neq{A'}\subset\Rd$ is a summand of the compact convex set  $\emptyset\neq{A}\subset\Rd$ if and only if $(A\sim A')+A'=A$, where $A\sim A':=\cap_{\mathbf{a}'\in A'}(A-\mathbf{a}')$. This implies that
$(B_{\K}[\oo , r]\sim X_{\K}^r)+X_{\K}^r=B_{\K}[\oo , r]$. Finally, we are left to observe that $B_{\K}[\oo , r]\sim X_{\K}^r=\cap_{\x\in X_{\K}^r}(B_{\K}[\oo , r]-\x)=
-\cap_{\x\in X_{\K}^r}B_{\K}[\x , r]=-{\rm conv}_{r,\K}(X)$, finishing the proof of Lemma~\ref{special}.
\end{proof}

\begin{remark}\label{open problem on special norm}
It seems to be an open problem to characterize those Minkowski spaces $\Md_{\K}$ for which (\ref{intersection-5}) holds. Nevertheless Theorem 8 of \cite{MRS} states that if (\ref{intersection-5}) holds in $\Md_{\K}$, then $\| \cdot \|_{\K}$ is a perfect norm (that is, every complete set is of constant width). For conditions equivalent to (\ref{intersection-5}) see Theorem 6 in \cite{MRS}.
\end{remark}

Clearly, the Brunn--Minkowski inequality (\cite{BuZa}, \cite{Sc}) combined with Lemma~\ref{special} yields

\begin{corollary}\label{special-corollary}
Let $d>1$and $r>0$ be given and let $\Md_{\K}$ possess a generating unit ball. If $X_{\K}^r\neq\emptyset$, then
\begin{equation}\label{intersection-6}
V_d(X_{\K}^r)^{\frac{1}{d}}+V_d({\rm conv}_{r,\K}(X))^{\frac{1}{d}}\leq r V_d(\K)^{\frac{1}{d}} . 
\end{equation}
\end{corollary}

Next, observe that based on (\ref{intersection-2}) we have $\emptyset\neq P_{\K}^r=\left(P_{\frac{\lambda}{2}}^{\K}\right)_{\K}^{r+\frac{\lambda}{2}}$ and so, Corollary~\ref{special-corollary} yields
$$V_d(P_{\K}^r)=V_d\left(\left(P_{\frac{\lambda}{2}}^{\K}\right)_{\K}^{r+\frac{\lambda}{2}}\right)\leq$$
\begin{equation}\label{intersection-7}
\left[\left(r+\frac{\lambda}{2}\right)V_d(\K)^{\frac{1}{d}}-V_d\left({\rm conv}_{r+\frac{\lambda}{2}, \K}\left(P_{\frac{\lambda}{2}}^{\K}\right)\right)^{\frac{1}{d}}\right]^d\leq
\left(r-(N^{\frac{1}{d}}-1)\frac{\lambda}{2} \right)^dV_d(\K),
\end{equation}
where in the last inequality we have used the fact that $\{B_{\K}\left[\p_i, \frac{\lambda}{2}\right]\ |\ 1\leq i\leq N\}$ is a packing in $\Rd$ and therefore 
$V_d\left({\rm conv}_{r+\frac{\lambda}{2}, \K}\left(P_{\frac{\lambda}{2}}^{\K}\right)\right)\geq N\left(\frac{\lambda}{2}\right)^dV_d(\K)$. Finally, observe that $N\geq 3^d$ implies
$\left(r-(N^{\frac{1}{d}}-1)\frac{\lambda}{2} \right)^dV_d(\K)\leq \left(r-\lambda\right)^dV_d(\K)<\left(r-\frac{d}{d+1}\lambda\right)^dV_d(\K)$. This inequality combined with (\ref{intersection-3}) and (\ref{intersection-7}) completes the proof of Theorem~\ref{intersection}.

\section{Proof of Remark~\ref{intersection-remark}}

Assume that the balls of $\Md_{\K}$ are generating sets in $\Rd$ and $\emptyset\neq A\subset\Rd$ is a compact set with $V_d(A)>0$. If $A_{\K}^r\neq\emptyset$, then Corollary~\ref{special-corollary} implies in a straightforward way that
\begin{equation}\label{intersection-remark-1}
V_d(A_{\K}^r)\leq\left[ r V_d(\K)^{\frac{1}{d}}-V_d({\rm conv}_{r,\K}(A))^{\frac{1}{d}} \right]^d\leq \left[ r V_d(\K)^{\frac{1}{d}}-V_d(A)^{\frac{1}{d}} \right]^d=V_d(B_{\K}[\oo, r-r_{\K}(A)]) .
\end{equation}
Thus, indeed (\ref{Blaschke-Santalo}) holds, that is, the balls of $\Md_{\K}$ are volumetric maximizers for $r$-ball bodies in $\Md_{\K}$.

Finally, assume that the balls of $\Md_{\K}$ are volumetric maximizers for $r$-ball bodies in $\Md_{\K}$. Moreover, assume that $Q:=\{\q_1, \dots , \q_N\}\subset\Rd$ is a uniform contraction of $P:=\{\p_1, \dots ,\p_N\}\subset\Rd$ with separating value $\lambda$ in $\Md_{\K}$ and $r>0, \lambda>0, d>1$, $N\geq 3^d$. We follow closely the proof of Theorem~\ref{intersection}. Thus, we clearly have (\ref{intersection-2}) and (\ref{intersection-3}). Next, observe that based on (\ref{intersection-2}) we have $\emptyset\neq P_{\K}^r=\left(P_{\frac{\lambda}{2}}^{\K}\right)_{\K}^{r+\frac{\lambda}{2}}$. As $\{B_{\K}\left[\p_i, \frac{\lambda}{2}\right]\ |\ 1\leq i\leq N\}$ is a packing in $\Rd$ therefore $r_{\K}\left(P_{\frac{\lambda}{2}}^{\K}\right)=N^{\frac{1}{d}}\frac{\lambda}{2}$ and so, (\ref{Blaschke-Santalo}) yields
\begin{equation}\label{intersection-remark-2}
V_d\left(\left(P_{\frac{\lambda}{2}}^{\K}\right)_{\K}^{r+\frac{\lambda}{2}}\right)\leq V_d\left(B_{\K}\left[\oo, \left(r+\frac{\lambda}{2}\right) -N^{\frac{1}{d}}\frac{\lambda}{2}\right]\right)=\left(r-(N^{\frac{1}{d}}-1)\frac{\lambda}{2} \right)^dV_d(\K) .
\end{equation}
Hence, (\ref{intersection-3}) and (\ref{intersection-remark-2}) imply (\ref{main-2}) in a straightforward way. 

\section{Proof of Theorem~\ref{Euclidean}}

\subsection{Proof of Part (i)}

We follow the above proofs of Theorem~\ref{union} and Remark~\ref{union-conv}. Clearly,
\begin{equation} \label{union-222}
{\rm diam}\left({\rm conv}(Q_r)\right)\leq \lambda+2r ,
\end{equation}
where ${\rm diam}\left(\cdot\right):={\rm diam}_{\B^d[\oo,1]}\left(\cdot\right)$. Now, recall that among all convex bodies of given diameter in $ \Ee^d$ precisely the balls have the greatest $k$-th quermassintegral for $0\leq k<d-1$ (\cite{Sc}, p. 335), that is, for any compact set $\emptyset\neq A\subset \Ee^d$ and $0\leq k<d$ we have 
\begin{equation}\label{union-2222}
W_k\left({\rm conv}(A)\right)\leq W_k\left(\B^d\left[\oo,\frac{{\rm diam}(A)}{2}\right]\right)=\left(\frac{{\rm diam}(A)}{2}\right)^{d-k}W_k\left(\B^d[\oo,1]\right)=\left(\frac{{\rm diam}(A)}{2}\right)^{d-k}\omega_d .
\end{equation}
Hence, (\ref{union-222}) and (\ref{union-2222}) yield that one can replace (\ref{union-22}) by the following inequality for $0\leq k<d$:
\begin{equation}\label{union-22222}
W_k\left({\rm conv}(Q_r)\right)\leq \left(r+\frac{\lambda}{2}\right)^{d-k}\omega_d .
\end{equation}
Next recall that among convex bodies of given (positive) volume in $ \Ee^d$ precisely the balls have the smallest $k$-th quermassintegral for any $0<k<d$ (\cite{Sc}, p. 335). This statement combined with (\ref{union-44}) (which has been derived under the assumption $0<\frac{\lambda}{2}\leq r$) implies the following inequality for $0\leq k<d$:
\begin{equation}\label{union-222222}
W_k\left({\rm conv}(P_r)\right)\geq   W_k\left(\B^d\left[\oo,r+(N^{\frac{1}{d}}-1)\frac{\lambda}{2}\right]\right)=\left(r+(N^{\frac{1}{d}}-1)\frac{\lambda}{2} \right)^{d-k}\omega_d .
\end{equation}    
Finally, $N\geq 2^d$, (\ref{union-22222}) and (\ref{union-222222}) complete the proof of (\ref{main-22-conv-Euclidean}). 

So, we are left to prove (\ref{main-22-Euclidean}). The proof that follows is an extension of the proof of (\ref{main-22-conv-Euclidean}). First, recall that Ohmann \cite{Oh54} proved the inequality  ${\omega_d}^{d-1-k}W_k(A)\leq \left({W_{d-1}(A)}\right)^{d-k}$ for any compact set $\emptyset\neq A\subset \Ee^d$ and $0\leq k<d$ with equality for balls. This result applied to $A=Q_r$ and ${\rm diam}(Q_r)\leq \lambda+2r $ yield that
\begin{equation}\label{extra-1}
W_k(Q_r)\leq\frac{1}{{\omega_d}^{d-1-k}}\left({W_{d-1}\left(\B^d\left[\oo,r+\frac{\lambda}{2}\right]\right)}\right)^{d-k}=\left(r+\frac{\lambda}{2}\right)^{d-k}\omega_d
\end{equation}
holds for $0\leq k<d$. Second, according to another result of Ohmann \cite{Oh52} the inequality $W_k(A)\geq W_k(\B^d[\oo,r(A)])$ holds for any compact set $\emptyset\neq A\subset \Ee^d$ and $0\leq k<d$, where the volumetric radius $r(A)$ of $A$ is defined by $V_d(\B^d[\oo,r(A)]):=V_d(A)$. If we apply this inequality to $A=Q_r$ and combine it with (\ref{union-4}), then we get that
\begin{equation}\label{extra-2}
W_k\left(P_r\right)\geq   W_k\left(\B^d\left[\oo,r+(N^{\frac{1}{d}}-1)\frac{\lambda}{2}\right]\right)=\left(r+(N^{\frac{1}{d}}-1)\frac{\lambda}{2} \right)^{d-k}\omega_d 
\end{equation}
holds for $0\leq k<d$. Thus, $N\geq 2^d$, (\ref{extra-1}), and (\ref{extra-2}) finish the proof of (\ref{main-22-Euclidean}).

\subsection{Proof of Part (ii)}

Recall that  $P:=\{\p_1,\dots ,\p_N\}\subset\Ee^d$ such that $ 0<\lambda\leq |\p_i-\p_j|\ {\rm holds\ for\ all} \ 1\leq i<j\leq N$, where $N\geq {2.359}^d$ with $d$ being sufficiently large.
We denote the circumradius of a set $X\subseteq\Ee^d$, $d>1$ by ${\rm cr} (X)$, which is defined by ${\rm cr}(X):=\inf\{r\ |\ X\subseteq \B^d[\x, r]\ {\rm for\ some}\ \x\in\Ee^d\}$.

\begin{lemma}\label{Bezdek-Kabatiansky-Levenshtein}
$\sqrt{\frac{2d}{d+1}}\left(\frac{\lambda}{2}\right)<0.7865\cdot\lambda<{\rm cr} (P)$, where $d\geq d_0$ with a large universal constant $d_0>0$ and ${\rm card}(P)=N\geq 2.359^d$.
\end{lemma}

\begin{proof}
First, we note that $\B^d[\p_1,\frac{\lambda}{2}],\dots ,\B^d[\p_N,\frac{\lambda}{2}]$ are pairwise non-overlapping in $\Ee^d$. Thus, the Lemma of \cite{Be02} and $N\geq 2.359^d$ imply that
\begin{equation}\label{Bezdek-1}
\frac{ {2.359}^d\left(\frac{\lambda}{2}\right)^d}{({\rm cr}(P)+\lambda)^d} \leq\frac{N\left(\frac{\lambda}{2}\right)^d}{({\rm cr}(P)+\lambda)^d}<\frac{V_d\left(\cup_{i=1}^N \B^d[\p_i,\frac{\lambda}{2}]\right)}{V_d\left(\cup_{i=1}^N \B^d[\p_i,\lambda] \right)}\leq \delta_d,
\end{equation}
where $\delta_d$ stands for the largest density of packings of congruent balls in $\Ee^d$. Second, recall that Kabatiansky and Levenshtein (\cite{CoZh}) have shown that 
\begin{equation}\label{Kab-Lev}
\delta_d<2^{-0.599d}
\end{equation}
holds for sufficiently large $d$ say, for $d\geq d_0$, where $d_0>0$ is a large universal constant. Hence, the statement follows from (\ref{Bezdek-1}) and (\ref{Kab-Lev}) in a straightforward way.
\end{proof}


If $r\leq {\rm cr(P})$, then $V_d(P^r)=V_d(\emptyset)=0$ and so, $V_d(P^r)\leq V_d(Q^r)$, i.e., (\ref{main-2-Euclidean}) follows. Thus, for the rest of the proof we assume that
${\rm cr}(P)< r$, which together with Lemma~\ref{Bezdek-Kabatiansky-Levenshtein} implies
\begin{equation}\label{assumption}
\sqrt{\frac{2d}{d+1}}\left(\frac{\lambda}{2}\right)<0.7865\cdot\lambda<{\rm cr}(P)< r
\end{equation}
with $d\geq d_0$ and ${\rm card}(P)=N\geq 2.359^d$. Next, as Euclidean balls are generating sets therefore (\ref{intersection-7}) implies the following statement. (See also Lemma 2.6 of \cite{BeNa} and (18) in \cite{Be19}.) 

\begin{lemma}\label{Bezdek-Naszodi}
If $d>1$, $\lambda>0$, $r>0$, and ${\rm card}(P)=N>1$, then
$V_{d}\left(P^{r}\right)\leq V_{d}\left(\B^d\left[\mathbf{o},r-\left(N^{\frac{1}{d}}-1\right)\left(\frac{\lambda}{2}\right)\right]\right)$.
\end{lemma}
\noindent Here we follow the convention that if $r-\left(N^{\frac{1}{d}}-1\right)\left(\frac{\lambda}{2}\right)<0$, then $\B^d\left[\mathbf{o},r-\left(N^{\frac{1}{d}}-1\right)\left(\frac{\lambda}{2}\right)\right]=\emptyset$ with  $V_{d}(\emptyset)=0$.

The statement that follows is a strengthening of (\ref{intersection-3}) as well as of Lemma 2.2 in \cite{BeNa}, i.e., of (13) in \cite{Be19} and it can be derived from a volumetric inequality of Schramm \cite{Sc88} in a rather straightforward way. For the sake of completeness, recall that  $Q:=\{\q_1,\dots ,\q_N\}\subset\Ee^d$ such that $ |\q_i-\q_j|\leq \lambda \ {\rm holds\ for\ all} \ 1\leq i<j\leq N$, where $N\geq {2.359}^d$ with $d$ being sufficiently large.

\begin{lemma}\label{Schramm-applied}
$V_{d}\left(Q^{r}\right)\geq    V_{d}\left(\B^d\left[\mathbf{o}, \sqrt{ r^2-\frac{d-1}{d+1}\left(\frac{\lambda}{2}\right)^2}-\left(\frac{\lambda}{2}\right) \right]\right) $, where $d\geq d_0$ and $N\geq 2.359^d$.
 \end{lemma}
 
 \begin{proof}
 First, recall Theorem 2 of \cite{Sc88}.
 \begin{theorem}\label{Schramm}
 Let $K$ be a set of diameter $\sigma$ and circumradius $\rho$ in $\Ee^d$. If $\mu>\rho>0$, then
 \begin{equation}\label{Schramm-general}
 V_d(K^{\mu})\geq F\left(\mu, \rho, \frac{\sigma}{2}\right)^d\omega_d,
 \end{equation}
 where $F(\mu,\rho, x):=\sqrt{\mu^2-\rho^2+x^2}-x$, which is a positive, decreasing, and convex function of $x>0$.
  \end{theorem}
 Second, Jung's theorem (\cite{Ju}) implies that ${\rm cr} (Q)\leq\sqrt{\frac{2d}{d+1}}\left(\frac{\lambda}{2}\right)$ and (\ref{assumption}) guarantees that $\sqrt{\frac{2d}{d+1}}\left(\frac{\lambda}{2}\right)<r$. Hence, from this and (\ref{Schramm-general}), using the monotonicity of $F(\mu,\rho, x)$ in $x$ (resp., $\rho$), one obtains

\begin{equation}\label{Jung-applied}
V_{d}\left(\B^d\left[\mathbf{o}, \sqrt{ r^2-\frac{d-1}{d+1}\left(\frac{\lambda}{2}\right)^2}-\left(\frac{\lambda}{2}\right) \right]\right) =
V_d\left(\B^d\left[\mathbf{o}, F\left( r,\sqrt{\frac{2d}{d+1}}\left(\frac{\lambda}{2}\right),\left(\frac{\lambda}{2}\right)\right)\right]\right) \leq V_d\left(Q^r\right),
\end{equation}
which completes the proof of Lemma~\ref{Schramm-applied}.
 \end{proof}

Clearly, Lemma~\ref{Bezdek-Naszodi} and Lemma~\ref{Schramm-applied} imply that in order to show the inequality $V_d(P^r)\leq V_d(Q^r)$, it is sufficient to prove 
\begin{equation}\label{Bezdek-Naszodi-Schramm}
r-\left(N^{\frac{1}{d}}-1\right)\left(\frac{\lambda}{2}\right)\leq \sqrt{ r^2-\frac{d-1}{d+1}\left(\frac{\lambda}{2}\right)^2}-\left(\frac{\lambda}{2}\right).
\end{equation}
(\ref{Bezdek-Naszodi-Schramm}) is equivalent to
\begin{equation}\label{almost-final}
\left(\frac{2r}{\lambda}\right)-\sqrt{\left(\frac{2r}{\lambda}\right)^2-\frac{d-1}{d+1}}+2\leq N^{\frac{1}{d}}
\end{equation}
and obviously, (\ref{almost-final}) follows (via $d\geq d_0$ and $N\geq 2.359^d$) from
\begin{equation}\label{final}
\left(\frac{2r}{\lambda}\right)-\sqrt{\left(\frac{2r}{\lambda}\right)^2-1}+2\leq 2.359\ .
\end{equation}
Finally, as $f(x):=x-\sqrt{x^2-1}$ is a positive and decreasing function for $x>1$ and as (\ref{assumption}) guarantees that $1.573<\frac{2r}{\lambda}$ therefore (\ref{final})
follows from $1.573-\sqrt{1.573^2-1}+2=2.3587...< 2.359$. This completes the proof of Theorem~\ref{Euclidean}.


\bigskip


\noindent K\'aroly Bezdek \\
\small{Department of Mathematics and Statistics, University of Calgary, Canada}\\
\small{Department of Mathematics, University of Pannonia, Veszpr\'em, Hungary\\
\small{E-mail: \texttt{bezdek@math.ucalgary.ca}}

\end{document}